\documentclass[11pt,a4paper]{article}
\usepackage{color}
\usepackage[latin1]{inputenc}
\usepackage[english]{babel}
\usepackage{amsfonts}
\usepackage{amssymb}
\usepackage{amsmath}

\usepackage{hyperref}

\newtheorem{theo}{Theorem}[section]
\newtheorem{lem}[theo]{Lemma}
\newtheorem{prop}[theo]{Proposition}
\newtheorem{cor}[theo]{Corollary}
\newtheorem{defn}[theo]{Definition}
\newtheorem{rem}[theo]{Remark}
\newtheorem{assump}[theo]{Assumption}
\newtheorem{exa}[theo]{Example}
\newenvironment{proof}{{\bf \noindent Proof.} }{\hfill $\Box$  \smallskip}

\def\R{\mathbb{R}}
\def\C{\mathbb{C}}

\def\Z{\mathbb{Z}}
\def\N{\mathbb{N}}
\def\No{{\mathbb{N}_0}}
\def\Rn{{\mathbb{R}^n}}

\def\epsilon{\varepsilon}
\def\seqo#1#2{(#1_#2)_{#2 \in \No}}

\def\supp{{\rm supp}\;}

\begin{document}

\title{On generalized Besov and Triebel-Lizorkin spaces of regular distributions\footnote{Research partially supported by {\it FEDER} funds through {\it COMPETE}--Operational Programme Factors of Competitiveness (``Programa Operacional Factores de Competitividade'') and by Portuguese funds through the {\it Center for Research and Development in Mathematics and Applications} (University of Aveiro) and the Portuguese Foundation for Science and Technology (``FCT--Funda\c c\~ao para a Ci\^encia e a Tecnologia''), within project PEst-C/MAT/UI4106/2011 with COMPETE number FCOMP-01-0124-FEDER-022690.}}
\author{Ant\'onio M. Caetano\footnote{Center for R\&D in Mathematics and Applications, Department of Mathematics, University of Aveiro, 3810-193 Aveiro, Portugal,
{\tt acaetano@ua.pt} (corresponding author).} \hspace{1mm} and Hans-Gerd Leopold\footnote{Mathematisches Institut, Friedrich-Schiller-Universit\"at Jena, D-07737 Jena, Germany,
{\tt hans-gerd.leopold@uni-jena.de}.}
}
\date{}

\maketitle

\begin{abstract}
We establish conditions on the parameters which are both necessary and sufficient in order that Besov and Triebel-Lizorkin spaces of generalized smoothness contain only regular distributions. We also connect this with the possibility of embedding such spaces in some particular Lebesgue spaces.
\end{abstract}

\hspace*{4mm}{\noindent\footnotesize{\it Mathematics Subject
Classification 2000:}} 46E35, 42A55.


\section{Introduction}
\label{Introd}
The main aim of this paper is to describe completely, in terms of their parameters, when the generalized Besov and Triebel-Lizorkin spaces $B^{\sigma,N}_{p,q}(\Rn)$ and $F^{\sigma,N}_{p,q}(\Rn)$ contain only regular distributions. In other terms, we aim to characterize the relations
$$B^{\sigma,N}_{p,q}(\Rn) \subset L_1^{\rm loc}(\Rn)$$
and
$$F^{\sigma,N}_{p,q}(\Rn) \subset L_1^{\rm loc}(\Rn)$$
in terms of the behaviour of $\sigma$, $N$, $p$ and $q$.

Besides the intrinsic interest of such a question within the theory of those spaces, such a characterization might also be useful when calculating with distributions belonging to them, as the possibility of representing distributions by functions naturally leads to simplifications.

A final answer to such a question in the context of classical spaces $B^s_{p,q}(\Rn)$ and $F^s_{p,q}(\Rn)$ was given in \cite[Theorem 3.3.2]{ST95}:

\begin{theo}
\label{classical}
\begin{enumerate}
    \item Let $s \in \R$, $0<p< \infty$ and $0<q \leq \infty$. Then the following two assertions are equivalent: \\
    1.1 $\; F^s_{p,q} \subset L_1^{\rm loc}$ \\
    1.2 \begin{tabular}[t]{lll}
        either $\; 0<p<1$, & $s \geq n(\frac{1}{p}-1)$, & $0<q \leq \infty$ \\
        or $\quad 1 \leq p < \infty $, & $s>0$, &  $0<q \leq \infty$ \\
        or $\quad 1 \leq p < \infty $, & $s=0$, &  $0<q \leq 2$
\end{tabular}
    \item Let $s \in \R$, $0<p \leq \infty$ and $0<q \leq \infty$. Then the following two assertions are equivalent: \\
  2.1 $\; B^s_{p,q} \subset L_1^{\rm loc}$ \\
    2.2 \begin{tabular}[t]{lll}
        either $\; 0<p \leq \infty$, & $s > n(\frac{1}{p}-1)_+$, & $0<q \leq \infty$ \\
        or $\quad 0 < p \leq 1 $, & $s = n(\frac{1}{p}-1)$, &  $0<q \leq 1$ \\
        or $\quad 1 < p \leq \infty $, & $s=0$, &  $0<q \leq \min\{ p,2 \}$
\end{tabular}
\end{enumerate}
\end{theo}

The spaces of generalized smoothness $B^{\sigma,N}_{p,q}(\Rn)$ and $F^{\sigma,N}_{p,q}(\Rn)$ in which we intend to study the same problem are natural generalizations of the classical Besov and Triebel-Lizorkin spaces in the direction of generalizing the smoothness and the partition in frequency. Now, instead of $(2^{sj})_j$, for some $s \in \R$, the smoothness will be controlled by a general so-called admissible sequence $\sigma := (\sigma_j)_j$, whereas the splitting in frequency will also be controlled by an admissible sequence $N := (N_j)_j$ more general than the classical $(2^j)_j$.

Such spaces have some history:

Originally they were introduced by Goldman and Kalyabin in the middle of the seventies of the last century with the help of differences and general weight functions and on the basis of expansions in series of entire analytic functions, respectively. In both cases these function spaces were subspaces of $L_p(\Rn)$, $1 < p < \infty$, by definition, therefore the question under which conditions they can contain or not contain singular distributions was pointless.

Later on spaces of generalized smoothness appeared naturally by real interpolation with function parameter, were used to describe compact and limiting embeddings with the help of the finer tuning given by the smoothness parameter and showed up also in connection with generalized $d$-sets and $h$-sets (special fractals) and function spaces defined on them, as well as in probability theory as generalized Bessel potential spaces.

For a historical survey up to the end of 2000, see \cite{FaLe06}.

As can be noticed by comparing our main assertions in Theorems \ref{1_nec_suf_B} and \ref{1_nec_suf_F} below with the classical counterpart recalled in Theorem \ref{classical} above, it is not at all clear why the latter should generalize in that way. As a matter of fact, it was somewhat of a surprise to us that the characterization could be done in such a neat way, specially in the cases where a comparison between the numbers $p$, $q$ and 2 seemed to be in order. We stress that we get a characterization, and not mere sufficient conditions. The bulk of the work has, indeed, to do with the proof that the guessed conditions are necessary. The tools used there rely heavily on the useful Proposition \ref{Inverse Hoelder}, which we denote by ``a reverse H\"{o}lder's inequality result'', and on the consideration of suitable sets of extremal functions. These are, for most of the cases, inspired by the possibility of representing the elements of the functions spaces under study by means of infinite linear combinations of atoms. Nevertheless, for the tricky cases given by the last lines in Theorems \ref{1_nec_suf_B} and \ref{1_nec_suf_F} we had to resort to lacunary Fourier series (and standardization) for that effect (by the way, Theorem \ref{functionseq} might also have independent interest).

As a by-product of our main results, we also extend to our framework the classical result \cite[Cor. 3.3.1]{ST95}, which states that the Besov and Triebel-Lizorkin spaces of integrability parameter $p\not= \infty$ which are completely formed by regular distributions are exactly those which continuously embed in the Lebesgue spaces of power $\max\{1,p\}$ --- cf. Corollary \ref{corollary}.

\medskip

We would like to thank H. Triebel for asking us what could be answered with respect to the main question dealt with in this paper, which prompted us to this research, and for some helpful occasional discussions about it.

\section{Preliminaries}
\label{Prelim}





We start by fixing some general notation.

Since all the Besov and Triebel-Lizorkin spaces under consideration are spaces on $\Rn$, we shall omit the $\Rn$ from the notation.

Given any $r\in(0,\infty]$, we denote by $r'$ the number, possibly $\infty$, defined through the expression $\frac{1}{r'} := \Big( 1 - \frac{1}{r} \Big)_+$; in the particular case when $1 \leq r \leq \infty$, $\, r'$ is the same as the conjugate exponent usually defined through $\, \frac{1}{r}+\frac{1}{r'}=1$.

The symbol $\hookrightarrow$ is used for continuous embedding from one space into other.

Unimportant positive constants might be denoted generically by the same letter, usually $c$, with additional indices to distinguish them in case they appear in the same or close expression.

Before introducing the spaces we want to consider, we define and make some comments about the type of sequences which will be used as parameters.

 \begin{defn}
 A sequence  $\sigma =(\sigma _j)_{j\in\No}$, with $\sigma _j>0$,
 is called an admissible sequence if there are two constants
 $0<d_0 =d_0(\sigma )\le d_1=d_1(\sigma )<\infty$ such that
\begin{equation}
\label{sigma}
                  d_0\, \sigma _j\le \sigma _{j+1}\le d_1\sigma _j
   \quad \mbox{for any}\quad j\in\No .
\end{equation}
  \end{defn}

 \begin{defn}
 \label{equiv}
           Two admissible sequences $\sigma = (\sigma _j)_{j\in\No}$
           and $\tau =(\tau _j)_{j\in\No}$ are called equivalent if there
           exist constants $C_1$ and $C_2$ such that
           $$
              0< C_1\le \frac{\sigma _j}{\tau _j}\le C_2<\infty 
              \quad\mbox{for any}\quad j\in\No .
           $$
 \end{defn}

To illustrate the flexibility of (\ref{sigma}) we refer the reader to
some examples discussed in \cite{FaLe06} or \cite[Chap. 1]{Bri-tese}.

The following definition, of Boyd indices of a given admissible sequence, is taken from \cite{BrMo03}:

\begin{defn}
Let
$$
  \overline{\sigma}_j:=  \sup\limits _{k\ge 0}\frac{\sigma _{j+k}}{\sigma
_k}
\quad\mbox{and}\quad
 \underline{\sigma}_j:= \inf\limits _{k\ge 0}\frac{\sigma _{j+k}}{\sigma _k}, \quad j\in\No.
$$
Then
\begin{equation}
\label{Boyd} \nonumber
  \alpha _{\sigma}:=\inf \limits _{j\in\N} \frac{\log_2  \overline{\sigma}_j}{j} = \lim_{j \to \infty } \frac{\log_2  \overline{\sigma}_j}{j}
\quad\mbox{and}\quad
   \beta _{\sigma}:=\sup \limits _{j\in\N} \frac{\log_2  \underline{\sigma}_j}{j} = \lim _{j\to \infty} \frac{\log_2  \underline{\sigma}_j}{j}
\end{equation}
are the (upper and respectively lower) Boyd indices of the sequence $\sigma$.
\end{defn}

\begin{rem} \label{estimations sigma}
{\em Obviously it holds $$ 
\log_2 d_0 \le \beta_\sigma \le \alpha_\sigma \le \log_2 d_1
$$
and for each $\varepsilon > 0$ there exist constants
$c_{0,\varepsilon}
> 0 $ and $c_{1,\varepsilon}  > 0  $ such that
$$ c_{0,\varepsilon}\; 2^{(\beta_\sigma - \varepsilon)j} \le \sigma_j \le
c_{1,\varepsilon}\;
2^{(\alpha_\sigma + \varepsilon)j} \;\;.$$ 
}
\end{rem}

\begin{rem}
\label{equivBoyd}
{\em (i)
    It is easy to see that the Boyd indices of an admissible sequence
    $\sigma$ remain unchanged
    when replacing $\sigma$ by an equivalent
    sequence in the sense of Definition~\ref{equiv}.

(ii)
     Given an admissible sequence $\sigma$ with Boyd indices
     $\alpha _{\sigma}$ and $\beta _{\sigma}$ then it is possible
      to find for any $\varepsilon >0$
     a sequence $\tau$ which is equivalent to $\sigma$ with
      $d_0(\tau )=2^{\beta _{\sigma}-\varepsilon}$
     and $d_1(\tau )=2^{\alpha _{\sigma}+\varepsilon}$, i.e.
     \begin{equation}
             2^{\beta _{\sigma}-\varepsilon}\, \tau _j
     \le
        \tau _{j+1}
     \le  2^{\alpha _{\sigma}+\varepsilon} \tau _j
         \quad\mbox{for any}\quad j\in\No .
     \end{equation}}
\end{rem}

\begin{assump}
\label{assump}

 From now on we will denote
 $N=(N_j)_{j\in\No}$  a sequence of real positive numbers
 such that  there exist
 two numbers $1<\lambda _0\le \lambda _1$ with
 \begin{equation}
 \label{N}
        \lambda _{0}\, N_j \le N_{j+1}\le \lambda _1 N_j \quad
 \mbox{for any}\quad j\in \No .
 \end{equation}
 \end{assump}
 
 In particular $N$ is a so-called strongly increasing sequence
 --- compare Definition~2.2.1 and Remark~4.1.2 in \cite{FaLe06}.
We would like to point out that the
 condition $\lambda_0>1$ played a key role in \cite[Assumption 4.1.1]{FaLe06} in
 order to get atomic decompositions
 in function spaces of generalized smoothness.

Moreover we choose a natural number $\kappa_0$ in such a way that
$ 2 \le \lambda_0^{\kappa_0}$
and consequently
$ 2 N_j \le N_{k} ~~\mbox{for any } ~j,k \in \N_0 $ such that $j+\kappa_0 \le k$
holds. We will fix such a $\kappa_0$ in the following.


 \begin{defn}
 \label{def-decomp}
 For a fixed sequence
 $N=(N_j)_{j\in\N_0}$ as in Assumption \ref{assump},
 let $\Phi ^{N}$ be the collection of all
 function systems $\varphi^{N}=(\varphi_j^{N})_{j\in\No}$
 such that:\\[2mm]
{\bf (i)} 
   $\varphi_j^{N}\in C_0^{\infty}(\R^n ) \quad\mbox{and}\quad
                  \varphi _j^{N}(\xi )\ge 0 \quad\mbox{if}\quad \xi\in\R^n
                 \quad\mbox{for any}\quad j\in\No$ ;\\[3mm]                 
{\bf (ii)}
  \hspace*{5mm}
   $\supp \varphi _j^{N} \subset \{\xi\in\R^n \, :\, |\xi |\le  N_{j+\kappa_0}\},~~~~~~~j = 0, 1, . . ., \kappa_{0}-1$,\\[2mm] 
 \hspace*{5mm}
     $~\qquad \supp \varphi _j^{N} \subset \{\xi\in\R^n \, :\, N_{j-\kappa_0}\le |\xi |\le N_{j+\kappa_0}\}
   \quad\mbox{if}\quad j \ge \kappa_0$;\\[3mm]   
{\bf (iii)} ~~ for any $\gamma\in\N_0^n$ there exists a constant
$c_{\gamma}>0$
        such that for any $j\in\N_0$
       \begin{equation*}
           \label{fl-i3}
          | D^{\gamma}\varphi _j^{N}(\xi )|
 \le c_{\gamma}\, 
 (1+|\xi|^2)^{-|\gamma |/2}
         \quad\mbox{for any}\quad \xi\in\R^n ;
        \end{equation*}
 {\bf (iv)}~~~there exists a constant $c_{\varphi}>0$ such that
 \begin{equation*}
 \label{fl-i4}
            0< \sum\limits _{j=0}^{\infty}\varphi _j^{N}(\xi )
 =c_{\varphi}<\infty
             \quad\mbox{for any}\quad \xi\in\R^n   .
 \end{equation*}
 \end{defn}




 In what follows
 $\cal S $ stands for the Schwartz space of all complex-valued
 rapidly decreasing
 infinitely differentiable functions on $\mathbb{R}^n$ equipped with the usual
 topology, $\cal S '$ denotes its topological dual, the space
 of all tempered distributions on $\R^n$, and ${\cal F}$ and ${\cal F}^{-1}$ stand respectively for the Fourier transformation and its inverse.

Let $(\sigma_j)_{j\in\N_0}$ be an admissible
 sequence, $(N_j)_{j\in\N_0}$ be an admissible sequence satisfying
Assumption \ref{assump} and let $\varphi ^{N}\in\Phi ^{N}$.

\begin{defn}
            \label{defbf2}
  \vspace*{0.6cm}          {\bf (i)} Let $0<p\le \infty$ and $0 <q\le\infty$.
 The Besov space $B^{\sigma ,N}_{p,q}$ of generalized smoothness is defined as

 $$
 \left\{
 f\in {\cal S '}\,:\,
  \Big\|f\,|\, B^{\sigma ,N}_{p,q} \Big\| :=
    \Big(\sum\limits _{j=0}^{\infty}
            \sigma_j^q \, \|{\cal F}^{-1}\, (\varphi_j^N {\cal F}f)|L_p(\R^n)\|^q\,
    \Big)^{1/q} <\infty\right\}.
 $$

 {\bf (ii)}
             Let $0<p< \infty$ and $0 < q\le \infty$.
 The Triebel - Lizorkin space $F^{\sigma , N}_{p,q}$ of generalized smoothness is defined as

 $$
 \Big\{
 f\in{ \cal S '}\,:\,
 \|f\,|\, F^{\sigma ,N}_{p,q}\| :=
    \Big\|
    \Big(\sum\limits _{j=0}^{\infty}
            \sigma_j^q \, |{\cal F}^{-1}\, (\varphi_j^N {\cal F}f)(\cdot)|^q\,
    \Big)^{1/q}
   |\, L_p(\mathbb{R}^n) \Big\|
 <\infty\Big\}.
 $$
 In both cases one should use the usual modification when $q=\infty$.
 \end{defn}

 Both  $B^{\sigma ,N}_{p,q}$ and $F^{\sigma , N}_{p,q}$
 are Banach spaces which are independent of the choice of the
 system $(\varphi ^{N} )_{j\in \N_0}$, in the sense of equivalent quasi-norms.
 As in the classical case,
the embeddings
 ${\cal S} \hookrightarrow B^{\sigma , N}_{p,q}\hookrightarrow{\cal S '}$ and
 ${\cal S} \hookrightarrow F^{\sigma ,N}_{p,q}  \hookrightarrow{\cal S '}$ hold true for all admissible
 values of the parameters and sequences.
If $ p,q<\infty$ then ${\cal S}$
 is dense in $B^{\sigma , N}_{p,q}$ and in $F^{\sigma ,N}_{p,q}$.
Moreover, it is clear that $B^{\sigma ,N}_{p,p}=F^{\sigma ,N}_{p,p}$.

Note also that if $N_j=2^j$ and $\sigma = \sigma ^s := (2^{js})_{j\in\No}$
 with $s$ real,
 then
 the above spaces coincide with the usual function spaces $B^s_{p,q}$
 and $F^s_{p,q}$ on $\R^n$,  respectively.
  We shall use the simpler notation $B^s_{p,q}$ and $F^s_{p,q}$ in the more classical situation just mentioned.
  Even for general admissible $\sigma$, when $N_j=2^j$ we shall write simply $F^\sigma_{p,q}$ and $B^\sigma_{p,q}$ instead of $F^{\sigma ,N}_{p,q}$ and $B^{\sigma ,N}_{p,q}$, respectively.

We have the following relation between $B$ and $F$ spaces, the proof of which can be done similarly as in the classical case (cf. \cite[Prop. 2.3.2/2.(iii), p. 47]{Tri83}:

\begin{prop}\label{BF}
Let $0<p< \infty$, $0 < q \leq \infty$. Let $N$ and $\sigma$ be admissible sequences with $N$ satisfying also Assumption \ref{assump}. Then
$$
B^{\sigma,N}_{p,\min\{ p,q \}} \hookrightarrow F^{\sigma,N}_{p,q} \hookrightarrow B^{\sigma,N}_{p,\max\{ p,q \}}.
$$
\end{prop}

\bigskip

Of intrinsic interest are also embedding results involving such spaces. Here we present two which will, moreover, be of great service to us later on. In the case of Besov spaces, this is taken from \cite[The. 3.7]{CaFa04}:

\begin{prop}
\label{CF-t1}
Let $N=(N_j)_{j\in\No}$ be an admissible sequence as
in Assumption~\ref{assump}
 and let $\sigma = (\sigma _j)_{j\in\No}$
and $\tau =(\tau _j)_{j\in\No}$ be two further admissible sequences.
Let $0<p_1\leq p_2\le \infty$, $0 < q_1, q_2 \leq \infty$ and
 $\frac{1}{q^*} := \left( \frac{1}{q_2} - \frac{1}{q_1} \right)_+$.
If
\begin{equation}
\label{CF-eq3}
              \left (
          \sigma _j^{-1}\, \tau _j\, N_j^{n\left (\frac{1}{p_1}-\frac{1}{p_2}\right )}
              \right )_{j\in\No}\in \ell _{q^*}
\end{equation}
then             $\; \; B^{\sigma , N}_{p_1,q_1}\hookrightarrow B^{\tau , N}_{p_2,q_2}$.
\end{prop}

\bigskip


The following partial counterpart for the $F$-spaces (which will be enough for our purposes) can be proved similarly (cf. also \cite[Prop. 1.1.13.(iv),(vi)]{Mou01b}):

\begin{prop}
\label{Fcounterpart}
Let $N$ be an admissible sequence as in Assumption~\ref{assump} and let $\sigma$
and $\tau$ be two further admissible sequences. Let $0<p<\infty$, $0<q_1,q_2\leq \infty$ and  $\frac{1}{q^*} := \left( \frac{1}{q_2} - \frac{1}{q_1} \right)_+$.
If
\begin{equation}
\label{CF-eq3'}
              \left (
          \sigma _j^{-1}\, \tau _j\, \right )_{j\in\No}\in \ell _{q^*}
\end{equation}
then             $\; \; F^{\sigma , N}_{p,q_1}\hookrightarrow F^{\tau , N}_{p,q_2}$.
\end{prop}

\bigskip

We state now sufficient conditions, already known to us, in order that $B^{\sigma , N}_{p,q}$ and $F^{\sigma , N}_{p,q}$ contain only regular distributions.

\begin{prop}[{\cite[Corol 3.18]{CaFa04}}]
\label{BinLpbar}
Let $0<p\leq \infty$, $0 < q \leq \infty$. Let $N$ and $\sigma$ be admissible sequences with $N$ satisfying also Assumption \ref{assump}. If
    \[  \label{CF-eq2}
          \left (
                   \sigma _j^{-1}\, N_j^{n\left (\frac{1}{p}-1\right )_{+}}
           \right )_{j\in\No}  \in \ell _{q^{\prime}}
\]
then   $\; B^{\sigma ,N}_{p,q}\hookrightarrow L_{\max\{1,p\}}$.
\end{prop}

\begin{rem}
\label{BinL1loc}
{\em As an immediate consequence we get, with the hypotheses of the previous proposition, also the conclusion $\; B^{\sigma ,N}_{p,q}\subset L_1^{\rm loc}$.}
\end{rem}

\begin{prop}[{\cite[Sec. 4, Prop. 3]{CaLe05}}]
\label{sufcond1}
Let $0<p< \infty$, $0 < q \leq \infty$. Let $N$ and $\sigma$ be admissible sequences with $N$ satisfying also Assumption \ref{assump}. If
    \[\label{cond1}
        \left\{
            \begin{array}{lr}
                (\sigma_j^{-1} N_j^\delta)_{j \in \No} \in \ell_{p'},\; \mbox{ for some }\, \delta >0, & \mbox{ if }\, 1 \leq p < \infty \\
                (\sigma_j^{-1} N_j^{n(\frac{1}{p}-1)})_{j \in \No} \in \ell_{\infty}, & \mbox{ if }\, 0 < p < 1\,,
            \end{array}
            \right.
\]
then
    $\; F^{\sigma,N}_{p,q} \subset L_1^{{\rm loc}}$.
\end{prop}


\section{Preparatory results}


In order to deal with the main question formulated in this paper, we need to introduce some technical tools and derive some results which will be required later on.

\subsection{Standardization}

In our setting, standardization is the ability to identify our generalized spaces with spaces where $N_j$ has the classical form $2^j$.

Let $N$ and $\sigma$ be admissible sequences, $N$ satisfying also the Assumption \ref{assump} as before, and let $\kappa_0$ be the fixed natural number with 
 $ \lambda_0^{\kappa_0} \ge 2$~.
Define
    \begin{equation} \label{beta}
    \beta_j := \sigma_{k(j)}, \quad \mbox{with }\; k(j):= \min \{ k \in \No : 2^{j-1} \leq N_{k+\kappa_0} \}, \;\; j \in \No .
\end{equation}
Then we have that
    \[
    \mu_0 \beta_j \leq \beta_{j+1} \leq \mu_1 \beta_j, \quad j \in \No,
\]
with $\mu_0 = \min \{ 1, d_0^{\kappa_0} \}$, $\mu_1 = \max \{ 1, d_1^{\kappa_0} \}$.

Under these conditions we proved
in \cite[Theorem 1]{CaLe05} the following standardization:

\begin{theo} \label{standard}
Let $N$ and $\sigma$ be admissible sequences, $N$ satisfying also the Assumption \ref{assump}. 
Let, further, $0 < p,q \leq \infty$ (with $p \neq \infty$ in the $F$-case). Then
    \[
    F^{\sigma,N}_{p,q} = F^\beta_{p,q} \quad and \quad  B^{\sigma,N}_{p,q} = B^\beta_{p,q},
\]
where $\beta := \seqo{\beta}{j}$ is determined by (\ref{beta}).

\end{theo}

As a consequence of this we obtain in case $\sigma_j = \sigma_j^0 = 1$ for all $j \in \N_0$:

\begin{cor} \label{level zero}Let $(\sigma_j)_{j\in\N_0}$ and $(N_j)_{j\in\N_0}$ be as before and $0 < p,q \leq \infty$ (with $p \neq \infty$ in the $F$-case).
Then
\begin{equation}
B^{(1),N}_{p,q} = B^{\sigma^0,N}_{p,q}= B^0_{p,q}~~
\end{equation}
and
\begin{equation}
 F^{(1),N}_{p,q} = F^{\sigma^0,N}_{p,q}= F^0_{p,q}~~.
 \end{equation}
\end{cor}

This extends \cite[Theorem 3.1.7]{FaLe06} also to the F-spaces and to the case $0<p\le 1$. The corollary will be useful to prove the sufficiency  of the conditions in Theorems \ref{1_nec_suf_B} and \ref{1_nec_suf_F}.

One of the most significant ingredients in the proof of the following theorem, which is Lemma 1 in \cite{CaLe05} and will also be useful later on, was again the above standardization theorem.

\begin{theo}
\label{BFB}
Let $0<p_1<p<p_2\leq \infty$, $0 < q \leq \infty$. Let $N := \seqo{N}{j}$ and $\sigma := \seqo{\sigma}{j}$ be admissible sequences with $N$ satisfying also Assumption \ref{assump}. Let $\sigma'$ and $\sigma''$ be the admissible sequences defined respectively by
    \[
    \sigma'_j = N_j^{n(\frac{1}{p_1}-\frac{1}{p})}\sigma_j, \quad \sigma''_j = N_j^{n(\frac{1}{p_2}-\frac{1}{p})}\sigma_j, \quad j \in \No.
\]
Then
    \[
    B^{\sigma',N}_{p_1,u} \hookrightarrow F^{\sigma,N}_{p,q} \hookrightarrow B^{\sigma'',N}_{p_2,v}
\]
if, and only if, $\;0 < u \leq p \leq v \leq \infty$.
\end{theo}
%

\medskip

As we shall see, our main results will be established in terms of the behaviour of the sequences $\sigma$ and $N$. Sometimes it is useful to deal with the case of general $N$ after having dealt with the more classical situation  when $N=(2^j)_{j\in \No}$, through standardization. The problem afterwards then might be that the criteria obtained are expressed in terms of $(\sigma_{k(j)}^{-1})_{j\in \No}$, for the $k(j)$ defined in \eqref{beta}, instead of the original sequence $(\sigma_j^{-1})_{j \in \No}$. This difficulty can, however, be circumvented by the following observations.

\begin{rem}
\label{propk(j)}
{\em From the definition of $k(j)$ the following two properties easily follow:
\begin{description}
\item[(i)] For $\kappa_0$ the fixed natural number such that $\lambda_0^{\kappa_0} \geq 2$, it holds
$$
k(j+1) \leq k(j) + \kappa_0, \quad j \in \No.
$$
\item[(ii)] There is $c_0 \in \N$ such that
$$
k(j+c_0) > k(j), \quad j \in \No;
$$
for example, $c_0=\kappa_1+j_0$, where $\kappa_1 \in \N$ satisfies $\lambda_1 \leq 2^{\kappa_1}$ and $j_0 \in \No$ is chosen such that $2^{j_0-1} > \lambda_1^{\kappa_0}N_0$.
\end{description}}
\end{rem}

\begin{prop}
\label{sigma-sigmak(j)}
Let $\sigma$ be an admissible sequence and $0 < r \leq \infty$. Let $k(j)$ be defined as in \eqref{beta}. Then
$$\sigma^{-1} \in \ell_r \quad \mbox{if, and only if,} \quad (\sigma_{k(j)}^{-1})_{j \in \No} \in \ell_r.$$
\end{prop}

\begin{proof}
We deal only with the main case when $0<r<\infty$. The case $r=\infty$ can be dealt with usual modifications.

Consider the numbers $\kappa_0$ and $c_0$ as in Remark \ref{propk(j)}.

On one hand,
\begin{eqnarray}
\label{lessthan}
\sum_{j=0}^\infty \sigma_{k(j)}^{-r} & = & \sum_{l=0}^\infty \sum_{m=0}^{c_0-1} \sigma_{k(lc_0+m)}^{-r} \nonumber\\
& = & \sum_{m=0}^{c_0-1} \sum_{l=0}^\infty \sigma_{k(lc_0+m)}^{-r} \nonumber \\
& \leq & \sum_{m=0}^{c_0-1} \sum_{j=0}^\infty \sigma_j^{-r} \nonumber \\
& = & c_0\, \sum_{j=0}^\infty \sigma_j^{-r},
\end{eqnarray}
where the inequality is justified by the fact that, for each fixed $m=0, \ldots, c_0-1$, $(\sigma_{k(lc_0+m)})_{l \in \No}$ is a subsequence of $\sigma$, as follows from Remark \ref{propk(j)}(ii).

On the other hand,
\begin{eqnarray}
\label{greaterthan}
\kappa_0\, \sum_{j=0}^\infty \sigma_{k(j)}^{-r} & = & \sum_{j=0}^\infty \sum_{m=0}^{\kappa_0-1} \sigma_{k(j)}^{-r} \nonumber \\
& \geq & c\, \sum_{j=0}^\infty \sum_{m=0}^{\kappa_0-1} \sigma_{k(j)+m}^{-r} \nonumber \\
& \geq & c\, \sum_{l=k(0)}^\infty \sigma_l^{-r},
\end{eqnarray}
where the first inequality is a direct consequence of the admissibility of $\sigma$ (with the factor $c$ depending on $\kappa_0$) and the second inequality comes from the fact that the term following each $\sigma_{k(j)+\kappa_0-1}^{-r}$ in the middle line above, being $\sigma_{k(j+1)}^{-r}$ is, by Remark \ref{propk(j)}(i), either the next term in the sequence $\sigma^{-r}$ or a term already considered before and that we can discard, turning the total sum smaller, though not smaller than the sum in the last line (because of Remark \ref{propk(j)}(ii)).

Combining \eqref{lessthan} and \eqref{greaterthan}, we get the required result.

\end{proof}

\medskip





\subsection{Atomic decompositions}
\label{atomicdec}


One of the tools we shall need is the atomic representation of functions in spaces of generalized smoothness.
In order to present the atomic decomposition theorem we first
set up some notation and terminology --- see also \cite[Section 4.4]{FaLe06}.

 Let  $\Z^n$ be  the lattice of all
 points in $\R^n$ with integer-valued components.

 If $\nu\in\No $ and $m=(m_1,...,m_n)\in \Z^n$
 we denote by $Q_{\nu m}$ the cube in $\R^n$ centred at
 $N_{\nu}^{-1}m= (N_{\nu}^{-1}m_1,...,N_{\nu}^{-1}m_n)$
 which has sides parallel to the axes and side length~$N_{\nu}^{-1}$.

 If $Q_{\nu m}$ is such a cube in $\R^n$
 and $c>0$ then $cQ_{\nu m}$ denotes the cube in $\R^n$ concentric with $Q_{\nu m}$
 and with side length $cN_{\nu}^{-1}$.

 \medskip


 \begin{defn}
 \label{def-21}
 {\bf (i)} Let $M\in \No$, $c^{*}>1$ and $\kappa >0$.
        A function $\rho :\R^n\to \C$ which is $M$ times differentiable 
 (continuous if $M = 0$)
 is called an  $1_M$-$N$-atom if:
 \begin{equation}
 \label{312-1a}
      \supp\, \rho \subset c^{*}Q_{0m}
      \quad\mbox{for some}\quad m\in\Z^n ,
 \end{equation}
 \begin{equation}
 \label{312-1b}
 \left
     |D^{\alpha} \rho (x)\right |
   \le
      \kappa
  \quad\mbox{if}\quad |\alpha | \le M.
 \end{equation}

 {\bf (ii)} Let $\sigma =(\sigma _j)_{j\in\No}$ be an admissible
sequence, let
 $0<p\le\infty$, $M, L+1 \in \No$, $c^{*}>1$ and $\kappa >0$.
 A function $\rho :\R^n\to \C$ which is $M$ times differentiable 
 (continuous if $M = 0$)
 is called an  $(\sigma , p)_{M,L}$-$N$-atom if:
 \begin{equation}
 \label{312-2a}
 \supp \, \rho \subset c^{*}Q_{\nu m}\quad
 \mbox{for some}\quad \nu\in\N\, ,   m\in\Z^n ,
 \end{equation}
 \begin{equation}
 \label{312-2b}
                   \left | D^{\alpha}\rho (x) \right |
               \le
                   \kappa\,
                          \sigma _{\nu}^{-1} \,
                          N_{\nu}^{\frac{n}{p}+|\alpha|}
 \quad\mbox{if} \quad
             |\alpha |\le M ,
 \end{equation}
 \begin{equation}
 \label{312-2c}
 \int _{\R^n}x^{\gamma}\rho (x) dx=0\quad \mbox{if} \quad |\gamma |\le L.
 \end{equation}
 \end{defn}

 If the atom $\rho$ is located at $Q_{\nu m}$
 (that means $\supp\,\rho \subset c^{*} Q_{\nu m}$ with $\nu\in \No\, $,
 $m\in \Z^n$, $c^{*}>1$)
 then we will denote it by  $\rho _{\nu m}$.

 \medskip



 As in the classical case, the $N$-atoms
 (associated to the  sequence $N$) are normalised building blocks
 satisfying some moment conditions.

 The value of the number $c^{*}>1$ in (\ref{312-1a}) and
 (\ref{312-2a}) is unimportant. It simply makes clear that at the level
 $\nu$
 some controlled overlapping of the supports of $\rho _{\nu m}$
 must be allowed.

 The moment conditions (\ref{312-2c}) can be reformulated as
 $D^{\gamma}\widehat{\rho}(0)=0$ if $|\gamma |\le L$,
 which shows that a sufficiently strong decay of $\widehat{\rho}$ at the
 origin is required.
 If $L<0$ then (\ref{312-2c}) simply means that
 there are no moment conditions required.

 The reason for the normalising factor in (\ref{312-1b}) and (\ref{312-2b})
 is that then there exists a constant $c>0$, depending on $\kappa$, such that for all these atoms we
 have
 $\|\rho \, |\, B^{\sigma, N}_{p,q}\|\le c$ and
 $\|\rho \,|\, F^{\sigma , N}_{p,q} \|\le c$, provided $M$ and $L$ are large enough --- see Theorem \ref{theo-23} below. In \cite{FaLe06} $\kappa$ was fixed to $ 1$ but we can use any other  $\kappa$ to the effect of normalisation.

 \medskip


 If $\nu\in\No\,$, $m\in\Z^n$ and $Q_{\nu m}$ is a cube as above, let
 $\chi _{\nu m}$ be the characteristic function of $Q_{\nu m}$; if
 $0<p\le\infty$
 let
 $$
  \chi _{\nu m}^{(p)}= N_{\nu}^{n/p}\chi _{\nu m}
 $$
 (obvious modification if $p=\infty$)
 be the $L_p$-normalised characteristic function of $Q_{\nu m} $.

 \begin{defn}
 \label{def-22}
 Let $0<p\le\infty$, $0<q\le \infty$. Then:

 \noindent {\bf (i)} $b_{p,q}$ is the collection of all sequences
 $\lambda =\{\lambda _{\nu m}\in \C\, :\, \nu\in\No\, , m\in\Z^n\}$
 such that
 $$
   \|\lambda \,|\,b_{p,q}\| =
 \left ( \sum_{\nu=0}^\infty \left ( \sum_{m \in \Z^n} |\lambda _{\nu m}|^{p} \right )^{q/p} \right )
   ^{1/q}
 $$
 (with the usual modification if $p=\infty$ and/or $q=\infty$) is finite;

 \noindent {\bf (ii)}  $f_{p,q}^N$ is the collection of all sequences
 $\lambda =\{\lambda _{\nu m}\in\C\, :\, \nu \in\No\, , m\in\Z^n \}$
 such that
 $$
   \|\lambda \,|\, f_{p,q}^N\| =
 \left \| \left ( \sum_{\nu=0}^\infty \sum_{m \in \Z^n} |\lambda _{\nu m}\chi _{\nu m}^{(p)}(\cdot )|^{q}
 \right )^{1/q} \, |\, L_p \right \|
 $$
 (with the usual modification if $p=\infty$ and/or $q=\infty$) is finite.
 \end{defn}

 One can easily see that $b_{p,q}$ and $f_{p,q}^N$ are quasi-Banach spaces
 and using $\|\chi _{\nu m}^{(p)}\,|\,L_p\|=1$ it is clear
 that comparing $\|\lambda\,|\, b_{p,q}\|$ and $\|\lambda\,|\, f_{p,q}^N\|$
 the roles of the quasi-norms in $L_p$ and $l_q$ are interchanged.\\



 In \cite{FaLe06} it was proved the following atomic decomposition theorem.

 \begin{theo}
 \label{theo-23}
 Let $N=(N_j)_{j\in\No}$ be an admissible sequence from Assumption \ref{assump} with $\lambda _0 >1$
 and let
 $\sigma =(\sigma _j)_{j\in\No}$ be an admissible sequence.

 Let $0<p<\infty$, respectively $0<p\le\infty$, $0<q\le \infty$,
 and let $M$, $L+1 \in \No$ be such that
 \begin{equation}
 \label{33-4}
 M >\frac{\log _2d_1}{\log _2\lambda _0}
 \end{equation}
 and
 \begin{equation}
 \label{33-5}
 L >-1 +n\left ( \frac{\log _2\lambda _1}{\log _2\lambda _0}
                   \frac{1}{\min (1, p, q)} -1
        \right )
       -\frac{\log _2 d_0}{\log _2 \lambda _0}  ,
 \end{equation}
 respectively
 \begin{equation}
 \label{33-5-bis}
 L >-1 +n\left ( \frac{\log _2\lambda _1}{\log _2\lambda _0}
                   \frac{1}{\min (1, p)} -1
        \right )
       -\frac{\log _2 d_0}{\log _2 \lambda _0}.
 \end{equation}
 Then $g\in \cal S '$ belongs to $F^{\sigma , N}_{p,q}$,
 respectively to $B^{\sigma , N}_{p,q}$,
 if and only if  it can be represented as
 \begin{equation}
 \label{33-6}
 g=\sum_{\nu=0}^\infty\sum_{m \in \Z^n}\lambda _{\nu m}\rho _{\nu m}\;\; ,
 \end{equation}
 convergence being in $\cal S '$,
 where $\rho _{\nu m}$ are  $1_M$-$N$-atoms ($\nu =0$) or
 $(\sigma ,p)_{M,L}$-$N$-atoms ($\nu\in\N$) and
 $\lambda\in f_{p,q}^{N}$, respectively $\lambda\in b_{p,q}$, where
 $\lambda =\{\lambda _{\nu m}\, :\, \nu \in\No \, , m\in \Z^n\}$.

 \noindent Furthermore, for any fixed $c^{*}>1$, any fixed $\kappa >0$,
           and any $M$ and $L$ as above,
 {\em inf} $\|\lambda\,|\, f_{p,q}^{N}\|$, respectively
 {\em inf} $\|\lambda \,|\, b_{p,q}\|$,
 where the infimum is taken over all admissible representations~{\em (\ref{33-6})},
 is an equivalent quasi-norm in $F^{\sigma , N}_{p,q}$, respectively
 $B^{\sigma , N}_{p,q}$.
 \end{theo}

For further comments, remarks, examples related to the above theorem
we refer the interested reader to~\cite{FaLe06}. The use of arbitrary $\kappa > 0$ instead of $\kappa = 1$ changes only the equivalence constants for the quasi-norm.

For the sake of completeness, we mention now an observation from an unpublished manuscript of Farkas and Leopold from 2007.
\begin{rem}
\label{imp-rem}
{\em
     Let  $\widetilde{N}= (\widetilde{N_j})_{j\in\No}$
     be an admissible sequence as in~Assumption \ref{assump} 
     which is equivalent to the sequence~$N$.

     Let also $\widetilde{\sigma} = (\widetilde{\sigma _j})_{j\in\No}$
     be an admissible sequence equivalent to $\sigma$.

     It follows directly from Definition~\ref{def-21}
     that for arbitrary fixed $c^{*}>1$ and $\kappa > 0$ there exist
     $\widetilde{c^{*}}>1$ and $\widetilde{\kappa} > 0$
     such that
             any $1_M$-$N$-atom is an $1_M$-$\widetilde{N}$-atom
     and such that
             any $(\sigma , p)_{M,L}$-atom is an $(\widetilde{\sigma} , p)_{M,L}$-$\widetilde{N}$-atom
     with respect to the numbers $\widetilde{c^{*}}$ and $\widetilde{\kappa}$.

     Clearly    $\widetilde{c^{*}}$ and $\widetilde{\kappa}>0$  depend on $c^{*}$, $\kappa$, $M$, $p$
     and on the equivalence constants for the sequences $\sigma$ and $N$.
}
\end{rem}
Let us denote by $\alpha _{\sigma}$ and $\beta _{\sigma}$, respectively
$\alpha _{N}$ and $\beta _N$, the Boyd indices of $\sigma$ and $N$ respectively.

According to Remark~\ref{imp-rem}, Remark~\ref{equivBoyd} and taking into account
the definition of Boyd indices,
conditions~(\ref{33-4})-(\ref{33-5-bis})
can be reformulated and improved as
\begin{equation}
   \label{33-4-new}
   M>\frac{\alpha _{\sigma}}{\beta _{N}}
   \quad\mbox{(replacement for}\quad \mbox{(\ref{33-4}))}\, ,
\end{equation}
and
\begin{equation}
\label{33-5-new}
      L >-1 +n
              \left (
                     \frac{\alpha _N}{\beta _N}
                     \frac{1}{\min (1, p, q)} -1
              \right )
         -
            \frac{\beta _{\sigma}}{\beta _N}
       \quad\mbox{(replacement for}\quad \mbox{(\ref{33-5}))}\, ,
\end{equation}
 \begin{equation}
 \label{33-5-bis-new}
 L >-1 +n\left (
                   \frac{\alpha _N}{\beta _N}
                   \frac{1}{\min (1, p)} -1
        \right )
       -\frac{\beta _{\sigma}}{\beta _N}
       \quad\mbox{(replacement for}\quad \mbox{(\ref{33-5-bis}))}.
 \end{equation}

\begin{rem} \label{atomic}{\em
We will refer to the above theorem, with conditions (\ref{33-4-new})-(\ref{33-5-bis-new}),
as to the atomic decomposition theorem in function spaces of generalized smoothness.}
\end{rem}

\begin{rem}
{\em We would like to stress that in the ``if'' assertion of the atomic decomposition theorem the convergence of \eqref{33-6} in ${\cal S}'$ is not an assumption, but rather a consequence of the hypotheses that the coefficients belong to $f^N_{p,q}$ or $b_{p,q}$. The proof of this can be done with the help of the results in Remark \ref{estimations sigma} and by adapting to the general situation the corresponding proof of Corollary 13.9(i) in \cite{Tri97}, more detailed explained in Proposition 1.20 in \cite{Mou01a}.}
\end{rem}

\subsection{Lacunary Fourier series}


%

At one point we shall need a specific result about lacunary Fourier series, that is, of Fourier series of the form
$$
\sum_{j=1}^\infty b_j e^{i\lambda_jt}
$$
where $(\lambda_j)_j$ is some given sequence of positive integers for which there exists $q$ such that $\frac{\lambda_{j+1}}{\lambda_j} > q > 1$, $j\in \N$.

\bigskip

For the following result, check \cite[p. 204]{Kah64} and references therein.

\begin{prop} \label{Kah}
If  $\; \sum_{j=1}^\infty b_j e^{i\lambda_jt}$, with $(\lambda_j)_j$ as above, is the Fourier series of a function of $L_1([0,2\pi])$, then $(b_j)_j \in \ell_2$.
\end{prop}

\medskip

%
%

Related to this, we shall also need the following technical lemma of \cite[Lemma 5.5.2]{Abel-tese} and the theorem which we state and prove afterwards, though the proof follows along the same lines of a corresponding result in \cite[Theorem 4.2.1; see also Remark 4.2.2.(c)]{Abel-tese}.

\begin{lem}
\label{techlem}
Let $N=(2^j)_{j\in\No}$ and consider a function system $\varphi = (\varphi_j)_{j\in\No} \in \Phi^N$ as in Definition \ref{def-decomp} built in the following way: for each $j \in \N \setminus \{ 1 \}$, $\, \varphi_j = \varphi_1(2^{-j+1} \cdot)$, where, for some suitable $a>0$,  $\varphi_1 \in {\cal S}$ is chosen such that
$$
\varphi_1(\xi) + \varphi_1(2^{-1}\xi) = 1 \quad \mbox{if } \; 2 \leq |\xi| \leq 4,
$$

$$
\varphi_1(\xi) = 1 \quad \mbox{if } \; 2(1-a) \leq |\xi| \leq 2(1+a)
$$
and
$$
{\rm supp}\, \varphi_1 \subset \{ \xi \in \R^n : \, (1+a) \leq |\xi| \leq 4(1-a) \};
$$
$\varphi_0 \in {\cal S}$ is chosen so that $\varphi_0(\xi)+\varphi_1(\xi)=1$ if $|\xi| \leq 2$ and ${\rm supp}\, \varphi_0 \subset \{ \xi \in \R^n : \, |\xi| \leq 2 \}$.
Consider $e_1=(1,0,\ldots,0) \in \Rn$. Given $\zeta \in {\cal S}$, $(b_j)_{j\in\N} \subset \C$ with $|b_j| \leq c_r 2^{jr}$ for some $r>0$ and $k \in \N$, the function
\begin{equation}
\label{V_k}
V_k := \sum_{j=1}^\infty b_j \zeta(\cdot - 2^j e_1) \varphi_k
\end{equation}
is well-defined with convergence in ${\cal S}$ and, for any given $d>0$,
\begin{equation}
\label{liminS}
\lim_{k \to \infty} 2^{kd} (V_k -b_k \zeta(\cdot - 2^k e_1)) = 0 \quad \mbox{in } \, {\cal S}.
\end{equation}
\end{lem}

\begin{theo}
\label{functionseq}
Let $0<p \leq \infty$ ($0<p<\infty$ in the case of $F$-spaces), $0<q\leq \infty$ and $\sigma$ be admissible. Let $\psi \in {\cal S} \setminus \{ 0 \}$ and $(b_k)_{k \in \N} \subset \C$ with $|b_k| \leq c_r 2^{kr}$ for some $r>0$. Then
$$
W(x_1,\ldots,x_n) := \sum_{j=1}^\infty b_j e^{i2^jx_1}
\vspace{-3mm}
$$
converges in ${\cal S}'$ and
\begin{equation}
\label{equivBFlq}
\psi W \in B^\sigma_{p,q} \; \Leftrightarrow \; (\sigma_k b_k)_{k \in \N} \in \ell_q \; \Leftrightarrow \; \psi W \in F^\sigma_{p,q}.
\end{equation}
\end{theo}

\begin{proof}
The hypothesis on the sequence $(b_k)_{k \in \N}$ immediately guarantees that $W$ makes sense in ${\cal S}'$ and is indeed a periodic distribution on $\Rn$ (cf. \cite[section 3.2]{ScTr87}). Then it is a straightforward calculation to see that
\begin{eqnarray*}
{\cal F}(\psi W) & = & \sum_{j=1}^\infty b_j {\cal F}(\psi e^{i2^jx_1}) \\
& = & \sum_{j=1}^\infty b_j ({\cal F}\psi)*\delta_{(2^j,0,\ldots,0)} \\
& = & \sum_{j=1}^\infty b_j ({\cal F}\psi)(\cdot -2^j e_1),
\end{eqnarray*}
where $e_1$ stands for $(1,0,\ldots,0) \in \R^n$.

Considering a system $\varphi$ as in Lemma \ref{techlem}, then
$$
\varphi_k {\cal F}(\psi W) = \sum_{j=1}^\infty b_j ({\cal F}\psi)(\cdot -2^j e_1) \varphi_k
$$
can be taken as the $V_k$ in \eqref{V_k}, $k \in \N$, for the choice $\zeta = {\cal F}\psi$. Therefore the conclusion \eqref{liminS} reads here as
$$
\lim_{k \to \infty} 2^{kd} (\varphi_k {\cal F}(\psi W) -b_k {\cal F}(\psi e^{i2^kx_1})) = 0 \quad \mbox{in } \, {\cal S},
$$
where $d>0$ is at our disposal. Applying the inverse Fourier transformation we get
\begin{equation}
\label{usedforBandF}
\lim_{k \to \infty} 2^{kd} ({\cal F}^{-1}(\varphi_k {\cal F}(\psi W)) -b_k \psi e^{i2^kx_1}) = 0 \quad \mbox{in } \, {\cal S}
\end{equation}
and, using ${\cal S} \hookrightarrow L_p$, also
\begin{equation}
\label{limLp}
\lim_{k \to \infty} 2^{kd} \| {\cal F}^{-1}(\varphi_k {\cal F}(\psi W)) -b_k \psi e^{i2^kx_1} | L_p \| = 0.
\end{equation}

Notice now that (with the usual modification in the case $q=\infty$) we have
\begin{eqnarray}
\lefteqn{\Big( \sum_{k=1}^\infty \sigma_k^q  \| b_k \psi e^{i2^kx_1} | L_p \|^q \Big)^{1/q}} \nonumber \\
& \leq & c\, \Big( \sum_{k=1}^\infty \sigma_k^q \| {\cal F}^{-1}(\varphi_k {\cal F}(\psi W)) | L_p \|^q \Big)^{1/q}  \label{quasitri} \\
&   &  + \;  c\, \Big( \sum_{k=1}^\infty \sigma_k^q \| {\cal F}^{-1}(\varphi_k {\cal F}(\psi W)) -b_k \psi e^{i2^kx_1} | L_p \|^q  \Big)^{1/q} \nonumber
\end{eqnarray}
and a corresponding estimation obtained by interchanging the roles of $b_k \psi e^{i2^kx_1}$ and ${\cal F}^{-1}(\varphi_k {\cal F}(\psi W))$. Since the last term in \eqref{quasitri} can be estimated from above by
$$
\Big( \sum_{k=1}^\infty \sigma_k^q 2^{-kdq}  \Big)^{1/q} \sup_{k \in \N} 2^{kd} \| {\cal F}^{-1}(\varphi_k {\cal F}(\psi W)) -b_k \psi e^{i2^kx_1} | L_p \|
$$
and, from \eqref{sigma}, $\sigma_k \leq \sigma_0 2^{k \log_2 d_1}$, by choosing $d>\log_2 d_1$ we get, also with the help of \eqref{limLp}, that the above expression is finite and therefore, from \eqref{quasitri} and the corresponding estimate referred to above,
\begin{eqnarray*}
\sum_{k=1}^\infty \sigma_k^q \| {\cal F}^{-1}(\varphi_k {\cal F}(\psi W)) | L_p \|^q &  & \mbox{is }\, \mbox{finite}
\end{eqnarray*}

\noindent if, and only if,
\begin{eqnarray*}
\sum_{k=1}^\infty \sigma_k^q  \| b_k \psi e^{i2^kx_1} | L_p \|^q & & \mbox{is }\, \mbox{finite}.
\end{eqnarray*}
That is, and after simplifying the last expression (taking also into consideration the hypothesis $\psi \in {\cal S} \setminus \{ 0 \}$),
$$
\psi W \in B^\sigma_{p,q}  \quad \mbox{if, and only if,} \quad  (\sigma_k b_k)_{k \in \N} \in \ell_q.
$$

As for $F^\sigma_{p,q}$, with $0<p,q<\infty$, we start by observing that from \eqref{usedforBandF} it follows, in particular, that for any $m \in \N$ and any $d>0$
$$
\lim_{k \to \infty} \sup_{x \in \Rn} \{(1+|x|)^m 2^{kd} |{\cal F}^{-1}(\varphi_k {\cal F}(\psi W)) -b_k \psi e^{i2^kx_1}|\} = 0.
$$
Then we have, pointwisely, with $d' > d$, that
\begin{eqnarray*}
\lefteqn{(1+|x|)^{mq} \sum_{k=1}^\infty 2^{kdq} |{\cal F}^{-1}(\varphi_k {\cal F}(\psi W)) -b_k \psi e^{i2^kx_1}|^q} \\
& \leq & \Big( \sum_{k=1}^\infty 2^{k(d-d')q} \Big) \big( \sup_{k\in \N} \sup_{x \in \Rn} (1+|x|)^m 2^{kd'} |{\cal F}^{-1}(\varphi_k {\cal F}(\psi W)) -b_k \psi e^{i2^kx_1}|\big)^q
\end{eqnarray*}
is finite and therefore the series of functions above converges pointwisely and, moreover,
\begin{equation}
\label{supdif}
\sup_{x \in \Rn} \{(1+|x|)^{mq} \sum_{k=1}^\infty 2^{kdq} |{\cal F}^{-1}(\varphi_k {\cal F}(\psi W)) -b_k \psi e^{i2^kx_1}|^q\} < \infty.
\end{equation}
Using now that $\sigma_k \leq \sigma_0 2^{k \log_2 d_1}$ --- cf. \eqref{sigma} --- and choosing $m \in \N$ large enough and $d\geq \log_2 d_1$ in \eqref{supdif}, we get that
\begin{eqnarray}
\label{finiteterm}
\lefteqn{\int_{\Rn} \Big( \sum_{k=1}^\infty \sigma_k^{q} |{\cal F}^{-1}(\varphi_k {\cal F}(\psi W)) -b_k \psi e^{i2^kx_1}|^q \Big)^{p/q}\, dx} \nonumber \\
& \leq & \int_{\Rn} (1+|x|)^{-mp}\, dx \\
& & \!\!\!\!\!\!\!\! \times \Big(\sup_{x \in \Rn} (1+|x|)^{mq} \sum_{k=1}^\infty 2^{k(\log_2d_1)q} |{\cal F}^{-1}(\varphi_k {\cal F}(\psi W)) -b_k \psi e^{i2^kx_1}|^q  \Big)^{p/q} \, < \, \infty \,. \nonumber
\end{eqnarray}
The counterpart of \eqref{quasitri} is now
\begin{eqnarray*}
\lefteqn{\Big\| \Big( \sum_{k=1}^\infty \sigma_k^q  | b_k \psi e^{i2^kx_1} |^q \Big)^{1/q}  | L_p \Big\| }  \\
& \leq & c\, \Big\| \Big( \sum_{k=1}^\infty \sigma_k^q | {\cal F}^{-1}(\varphi_k {\cal F}(\psi W)) |^q \Big)^{1/q} | L_p \Big\|   \\
&   &  + \;  c\, \Big\| \Big( \sum_{k=1}^\infty \sigma_k^q  | {\cal F}^{-1}(\varphi_k {\cal F}(\psi W)) -b_k \psi e^{i2^kx_1} |^q  \Big)^{1/q} | L_p \Big\|,
\end{eqnarray*}
and, again, a corresponding estimation obtained by interchanging the roles of $b_k \psi e^{i2^kx_1}$ and ${\cal F}^{-1}(\varphi_k {\cal F}(\psi W))$ also holds. Therefore, taking \eqref{finiteterm} into account,
\begin{eqnarray*}
\Big\| \Big( \sum_{k=1}^\infty \sigma_k^q | {\cal F}^{-1}(\varphi_k {\cal F}(\psi W)) |^q \Big)^{1/q} | L_p \Big\|  &  & \mbox{is }\, \mbox{finite} \end{eqnarray*}

\noindent if, and only if,
\begin{eqnarray*}
\Big\| \Big( \sum_{k=1}^\infty \sigma_k^q  | b_k \psi e^{i2^kx_1} |^q \Big)^{1/q}  | L_p \Big\|  & & \mbox{is }\, \mbox{finite}.
\end{eqnarray*}
That is, and after simplifying the last expression (taking also into consideration the hypothesis $\psi \in {\cal S} \setminus \{ 0 \}$),
$$
\psi W \in F^\sigma_{p,q}  \quad \mbox{if, and only if,} \quad  (\sigma_k b_k)_{k \in \N} \in \ell_q.
$$

We have been assuming, in this case of $F$-spaces, that both $p$ and $q$ are finite. However, with the usual modifications the preceding arguments also work out for $q=\infty$.
\end{proof}


\section{Main results}
\label{Main}


We start by considering a reverse H\"older's inequality result which will be used as a backbone for the proof of the necessity of most of the conditions in Theorems \ref{1_nec_suf_B} and \ref{1_nec_suf_F} below.



\begin{prop}
\label{Inverse Hoelder}
Let $0 < r \leq \infty$ and $(a_j)_{j \in \N}$, $(b_j)_{j \in \N} \subset \C$. If  $(a_j b_j)_{j \in \N}$ belongs to $ \ell_1$ for all sequences $(b_j)_{j \in \N}$ belonging to $ \ell_{r}$, then $(a_j)_{j \in \N} \in \ell_{r'}$.
\end{prop}

The case $1<r<\infty$ is contained in  \cite[The. 161, p. 120]{HLP}. The case $r=\infty$ is trivial (just take all $b_j$'s equal to 1), though something stronger is true, namely the conclusion still holds merely by drawing $(b_j)_{j \in \N}$ from $c_0$, as follows from \cite[The. 162(i), pp. 120-121]{HLP}. Finally, the case $0<r \leq 1$ (then $r'=\infty$) can be proved by contradiction. Indeed, assume  $(a_j)_{j \in \N} \not\in \ell_{\infty}$. Then for each natural number $l$ there exists an index $j_l > j_{l-1}$ such that $|a_{j_l}| \ge l^{\frac{1}{r}+1}$, where $j_0$ can, e.g., be taken equal to 1. Define
$$ b_j :=\left\{\begin{array}{lcr}
l^{-\frac{1}{r}-1} &\mbox{ if } &j=j_l~~~\\
                       0  &&\mbox{otherwise~~.}
                       \end{array}\right.  $$
Then  $(b_j)_{j \in \N} \in \ell_{r}$  but     $\sum_{j=1}^\infty |a_j| b_j = \infty$.

\bigskip
To prove the necessity of some conditions in the next theorem we will construct so-called extremal functions starting from a smooth basic function $\Phi$ with compact support and vanishing moment conditions, which we describe next:

\begin{prop} \label{basic function} For every $L\in\N$ and $\lambda_0 > 1$ there exist a $C^\infty$-function $\Phi$ on $\R^n$ and suitable positive constants $C_1\,,\,C_2\,$ and $C_3$, these constants depending only on $\lambda_0$ and $n$, such that $ C_1 < C_3 < \lambda_0 C_1~,$
$$ \Phi(x) \ge C_2 ~~\mbox{if}~~|x|_\infty \le C_1~,~~ \Phi(x) = 0 ~~\mbox{if}~~|x|_\infty \ge C_3~$$
and
$$ \int x^\gamma \Phi(x)~dx ~=~0 ~~\mbox{whenever}~~\gamma \in \N_0^n ~~\mbox{and}~~|\gamma|_\infty \le L~.$$
\end{prop}
A construction of such functions was described  in   \cite[Lem. 4.6]{CaFa04}.

\begin{theo}
\label{1_nec_suf_B}
Let $0<p, q \leq \infty$. Let $N$ and $\sigma$ be admissible sequences with $N$ satisfying also Assumption \ref{assump}.
The following are necessary and sufficient conditions for $\; B^{\sigma,N}_{p,q} \subset L_1^{\rm loc}$, where $\ell_{\frac{p\infty}{\infty-p}}$ should be understood as $\ell_p$:
\\[-1mm]
\begin{enumerate}
    \item $(\sigma_j^{-1} N_j^{n(\frac{1}{p}-1)})_{j \in \No} \in \ell_{q'}$, \quad in case $\; 0 < p \leq 1\;$ and $\; 0<q \leq \infty$;
    \item $(\sigma_j^{-1})_{j \in \No} \in \ell_\infty$, \quad in case $\; 1 < p \leq \infty\;$ and $\; 0 < q \leq \min\{ p,2 \}$;
    \item $(\sigma_j^{-1})_{j \in \No} \in \ell_{\frac{pq}{q-p}}$, \quad in case $\; 1< p \leq 2\;$ and $\; \min\{ p,2 \} < q \leq \infty$;
    \item $(\sigma_j^{-1})_{j \in \No} \in \ell_{\frac{2q}{q-2}}$, \quad in case $\; 2 < p \leq \infty\;$ and $\; \min\{ p,2 \} < q \leq \infty$.
\end{enumerate}

\end{theo}

\begin{proof}

(i) First we prove the sufficiency of the given conditions in each case.

In case 1, it follows directly from Remark \ref{BinL1loc}.

For each one of the remaining cases we use, in sequence, Proposition \ref{CF-t1}, Corollary \ref{level zero} and Theorem \ref{classical}. This explains why we can write, assuming the condition in each one of the cases, that in case 2

$$B^{\sigma,N}_{p,q} \hookrightarrow B^{(1),N}_{p,q} = B^0_{p,q} \subset L_1^{\rm loc},$$

\noindent in case 3

$$B^{\sigma,N}_{p,q} \hookrightarrow B^{(1),N}_{p,p} = B^0_{p,p} \subset L_1^{\rm loc}$$

\noindent and in case 4

$$B^{\sigma,N}_{p,q} \hookrightarrow B^{(1),N}_{p,2} = B^0_{p,2} \subset L_1^{\rm loc}.$$

\bigskip
(ii) Here we prove the necessity of the condition sated in case 1.

Let $L$ be chosen in dependency of $(N_j)_{j\in \No}$ and $(\sigma_j)_{j\in \No}$ by (\ref{33-5-bis}) or (\ref{33-5-bis-new}), respectively, and let $\Phi$ be a corresponding basic function depending on $L$, $n$ and $\lambda_0$ from Proposition \ref{basic function}.
Let $(\rho_j)_{j\in\N}$ be a sequence belonging to $\ell_q$ and
\begin{equation}
f^\rho(x) := \sum_{j=1}^\infty  |\rho_j| \sigma_j^{-1} N_j^{n/p} \Phi(N_jx),
\label{frho}
\end{equation}
convergence in  ${\cal S}'$.
For $x\not=0$ this is always a finite sum and  for each $j$ the functions $\sigma_j^{-1}N_j^{n/p}\Phi(N_j x)$ are $(\sigma,p)_{M,L}$-$N$-atoms located at $Q_{j0}$ in the sense of Definition \ref{def-21} and  Theorem \ref{theo-23}.
Then $f^\rho$ belongs to $\; B^{\sigma,N}_{p,q}\;$ and
$$ ||f^\rho|B^{\sigma,N}_{p,q}|| ~\le ~ c\, \| (\rho_j)_{j\in \N} | \ell_q \| .$$
Now we assume $\; B^{\sigma,N}_{p,q} \subset L_1^{\rm loc}$. Then
$$   \int_{|x|_\infty \le C_1 N_1^{-1}} |f^\rho(x)| \,dx   ~< ~\infty~~$$
and, actually, $f^\rho$ will also be given by \eqref{frho} in the pointwise sense a.e..
We will split part of the set $\{x: \,|x|_\infty \le C_1 N_1^{-1}\}$ in a non-overlapping way to obtain simple passages
$$P_m := \{x:\, C_3\lambda_0^{-1}N_m^{-1} \le |x|_\infty \le C_1N_m^{-1}\}\, ,$$
because on these passages we have
$$ \Phi(N_jx) \ge C_2 ~~~\mbox{if}~~~j\le m $$
and
$$ \Phi(N_jx) = 0 ~~~\mbox{if}~~~j > m ~~.$$

%

For each $k\in \N$ we have

\begin{eqnarray*}
\infty & > &  \int_{|x|_\infty \le C_1 N_1^{-1}} |f^\rho(x)| \,dx  \\
&\geq& \int_{C_3\lambda_0^{-1}N^{-1}_k \le |x|_\infty \le C_1 N_1^{-1}} \Big| \sum_{j=1}^\infty  |\rho_j| \sigma_j^{-1} N_j^{n/p} \Phi(N_jx) \Big| \,dx \\[2mm]
&\geq& \sum_{m=1}^{k}   \int_{C_3\lambda_0^{-1}N^{-1}_m \le |x|_\infty \le C_1 N_m^{-1}}  \Big| \sum_{j=1}^\infty  |\rho_j| \sigma_j^{-1} N_j^{n/p} \Phi(N_jx) \Big|\,dx \\[2mm]
&\geq& C_2 \sum_{m=1}^{k}  \int_{C_3\lambda_0^{-1}N^{-1}_m \le |x|_\infty \le C_1 N_m^{-1}}  \sum_{j=1}^m  |\rho_j| \sigma_j^{-1} N_j^{n/p} \,dx \\[2mm]
&\geq& C_2 \sum_{m=1}^k    |\rho_m| \sigma_m^{-1} N_m^{n/p} 2^n (C_1^n - C_3^n \lambda_0^{-n}) N_m^{-n} \\[2mm]
&=& c \sum_{m=1}^{k} |\rho_m| \sigma_m^{-1} N_m^{n(\frac{1}{p}-1)} ~~.
\end{eqnarray*}

The sum on the right-hand side is monotone increasing and the left-hand side is independent of $k$.
So we have
\begin{equation} \label{sigmaN}
\sum_{j=1}^{\infty}  |\rho_j| \sigma_j^{-1} N_j^{n(\frac{1}{p}-1)} ~< ~ \infty
\end{equation}
for any sequence   $(\rho_j)_{j\in\N} \in \ell_q$ if $\; B^{\sigma,N}_{p,q} \subset L_1^{\rm loc}$.
\\Now by  Proposition \ref{Inverse Hoelder} it follows

$$ (\sigma_j^{-1} N_j^{n(\frac{1}{p}-1)})_{j\in \N_0} \in \ell_{q'} ~~. $$

\bigskip
(iii) Now we prove the necessity of the conditions sated in cases 2 and 3.

\noindent
Let $(\gamma_j)_{j\in\N_0}$ be an arbitrary sequence 
belonging to $\ell_1$. 
 For technical reasons we consider now the sequence $(\tilde{\gamma}_j)_{j\in\N_0}$ with
\begin{equation} \label{gamma} \tilde{\gamma}_j := \max{(|\gamma_j|, 10^3 N_0^{-1}\lambda_0^{-j})} ~~~~~j = 0, 1,  \cdots ~~~.
\end{equation}
It is clear that $(\tilde{\gamma}_j)_{j\in\N_0}$ also belongs to $\ell_1$.
\\ Define
$$\kappa_0 := 0 ~~\mbox{and}~~\kappa_j:=\sum_{l=1}^j \tilde{\gamma}_l ~~ \mbox{for}~~j\in\N~~.$$
Then $\kappa_j > 0$ if $j\in \N$ and $\lim_{j\to \infty} \kappa_j = \kappa$, where $\kappa$ is equal to $||(\tilde{\gamma}_j)_{j\in\N_0}|\ell_1||$.
\\For all $j=1,2, \cdots$ put
$$R_j:=\{x=(x_1,x_2,...,x_n): \kappa_{j-1} < x_1 \le \kappa_j~,~0<x_i<1~~i=2,3,...n\} ~~
.$$
We obtain rectangles in $\R^n$ which become narrower in the $x_1$-direction.
Inside each $R_j$ we consider cubes $Q_{jm}$ of the type considered in the beginning of subsection \ref{atomicdec}. There exist  $M_j$ such cubes inside  $R_j$, centred in $N_j^{-1}m_r~,~j=1,\cdots,M_j$. Because of
$$ 10^3 N_j^{-1} \le 10^3 N_0^{-1} \lambda_0^{-j} \le  \tilde{\gamma}_j$$ and assuming, without loss of generality, that $N_0>2$, we have
$$M_j \sim N_j^{n-1} (\kappa_j-\kappa_{j-1})N_j = N_j^n\tilde{\gamma}_j~ ~.$$
In dependency of $(N_j)_{j\in \No}$ and $(\sigma_j)_{j\in \No}$ choose $L$ which fulfil (\ref{33-5-bis}) or (\ref{33-5-bis-new}), respectively.
Furthermore let $\Phi$ be a basic function depending on $L$, $n$ and $\lambda_0$ from Proposition \ref{basic function} and put $\tilde{\Phi}(x) := \Phi(2C_3 x)$.
 Let
\begin{equation}
h^\rho(x) := \sum_{j=1}^\infty \sum_{r=1}^{M_j} \rho_j \tilde{\Phi}(N_j(x-N_j^{-1}m_r))
\label{grhosum}
\end{equation}
(pointwise convergence) be a compactly supported function where  $(\rho_j)_{j\in\N}$ is an arbitrary sequence of non-negative numbers which will be specified later.
Notice that by construction for each $x\in \R^n$ in the double sum appears at most one summand which is not zero and that  $\sigma_j^{-1}N_j^{n/p}\tilde{\Phi}(N_j(x-N_j^{-1}m_r))$ are $(\sigma,p)_{M,L}$-$N$-atoms located at $Q_{jm_r}$ in the sense of Definition \ref{def-21} and  Theorem \ref{theo-23}.

If $(\rho^{(m)}_j\sigma_j N_j^{-n/p})_{j\in\No,m\in \Z^n} \in b_{p,q}$, where $\rho^{(m)}_j = \rho_j $ if $Q_{j,m} \subset R_j$ and $\rho^{(m)}_j = 0 $ otherwise, then the double sum in \eqref{grhosum} converges in ${\cal S}'$ to some $g^\rho$ which, by Theorem \ref{theo-23}, belongs to $B^{\sigma,N}_{p,q}$ and which, moreover, satisfies (assuming further that both $ p $ and $q$ are finite)

\begin{eqnarray*}||g^\rho|B^{\sigma,N}_{p,q}|| &\le & c \left(\sum_{j=1}^\infty \left(\sum_{r=1}^{M_j} |\rho_j \sigma_j N_j^{-\frac{n}{p}}|^p \right)^{q/p}\right)^{\frac{1}{q}} \\
& \sim & c \left(\sum_{j=1}^\infty  \rho_j^q \sigma_j^q N_j^{-\frac{nq}{p}} M_j^{\frac{q}{p}}\right)^{\frac{1}{q}}\\
& \sim & c \left(\sum_{j=1}^\infty  \rho_j^q \sigma_j^q N_j^{-\frac{nq}{p}}  N_j^{\frac{nq}{p}} \tilde{\gamma_j}^{\frac{q}{p}} \right)^{\frac{1}{q}}\\
& \sim & c \left(\sum_{j=1}^\infty  \rho_j^q \sigma_j^q  \tilde{\gamma}_j^{\frac{q}{p}} \right)^{\frac{1}{q}} ~~<~~\infty~~~.
\end{eqnarray*}
For each given sequence $(\gamma_j)_{j\in\N_0} \in \ell_1$ we choose
$$\rho_j^q := \sigma_j^{-q}  \tilde{\gamma}_j^{-\frac{q}{p} +1}$$
and obtain
$$ ||g^\rho|B^{\sigma,N}_{p,q}|| \le c ||(\tilde{\gamma}_j)_{j\in\N_0}|\ell_1||^{1/q} < \infty ~~~.$$
With this special choice of $(\rho_j)_{j\in\N}$ we also have
\begin{eqnarray*}
\int_{[0,\kappa]\times[0,1]^{n-1}} |h^\rho(x)| \,dx &=& \sum_{j=1}^\infty \int_{R_j}|h^\rho(x)| \,dx\\
&=&\sum_{j=1}^\infty \sigma_j^{-1} \tilde{\gamma}_j^{-\frac{1}{p}+\frac{1}{q}} \sum_{r=1}^{M_j} \int_{Q_{j,m_r}} |\tilde{\Phi}(N_j(x-N_j^{-1}m_r)|\,dx \\
&\sim& \sum_{j=1}^\infty \sigma_j^{-1} \tilde{\gamma}_j^{-\frac{1}{p}+\frac{1}{q}} N_j^{-n} M_j\\
&\sim&  \sum_{j=1}^\infty \sigma_j^{-1} \tilde{\gamma}_j^{-\frac{1}{p}+\frac{1}{q}}  N_j^{-n} N_j^{n} \tilde{\gamma}_j\\
&\sim&  \sum_{j=1}^\infty \sigma_j^{-1} \tilde{\gamma}_j^{1-\frac{1}{p}+\frac{1}{q}} ~~~,
\end{eqnarray*}
where the equivalence constants might depend on $\Phi$.
Now we assume $\; B^{\sigma,N}_{p,q} \subset L_1^{\rm loc}$. Then $h^\rho$ and $g^\rho$ coincide a.e. and for every sequence $(\gamma_j)_{j\in\N_0} \in \ell_1$
$$\sum_{j=1}^\infty \sigma_j^{-1} \tilde{\gamma}_j^{1-\frac{1}{p}+\frac{1}{q}} \sim \int_{[0,\kappa]\times[0,1]^n} |g^\rho(x)| \,dx  < \infty ~~.$$
Moreover by (\ref{gamma})  $$ \sum_{j=1}^\infty \sigma_j^{-1} |{\gamma}_j|^{1-\frac{1}{p}+\frac{1}{q}}
\le \sum_{j=1}^\infty \sigma_j^{-1} \tilde{\gamma}_j^{1-\frac{1}{p}+\frac{1}{q}} $$
whenever $1-\frac{1}{p}+\frac{1}{q} > 0$. But this is the case if $1 < p$.

Therefore $\; B^{\sigma,N}_{p,q} \subset L_1^{\rm loc}$ implies $\sum_{j=1}^\infty \sigma_j^{-1} |\gamma_j|^{1-\frac{1}{p}+\frac{1}{q}} < \infty$ for all sequences $(\gamma_j)_{j\in\N_0} \in \ell_1$.
But this is equivalent to  $\sum_{j=0}^\infty \sigma_j^{-1} |\beta_j| < \infty$ for all sequences $(\beta_j)_{j\in\N_0} \in \ell_r$, where $\frac{1}{r} = 1-\frac{1}{p}+\frac{1}{q} > 0 $.
 Then it
follows $(\sigma_j^{-1})_{j\in\N_0} \in \ell_{r'}$ by Proposition \ref{Inverse Hoelder}.
In case $ 1<r<\infty $ (this is equivalent to $p<q$)  we get $\frac{1}{r'} = \frac{1}{p}-\frac{1}{q}$  and in case $0<r\le 1$ (this is equivalent to $q \le p$) we have $r'=\infty$.
Consequently we obtain that $\; B^{\sigma,N}_{p,q} \subset L_1^{\rm loc}$ implies
\begin{eqnarray*}  (\sigma_j^{-1})_{j\in\N_0} \in \ell_{\frac{pq}{q-p}} &\mbox{if}& 1 <  p < q < \infty\, ,\\
(\sigma_j^{-1})_{j\in\N_0} \in \ell_{\infty}~~ &\mbox{if}& 1<p< \infty ~~\mbox{and}~~0 < q \le p~.
\end{eqnarray*}
Adapting the above arguments to the cases where $p$ or $q$ are infinite, we get the same conclusions as long as we interpret $\ell_{\frac{p\infty}{\infty-p}}$ as $\ell_p$.

(iv) Finally we prove the necessity of the condition sated in case 4.

\noindent
Let $\; B^{\sigma,N}_{p,q}$ be given. Then by Theorem \ref{standard} we find a sequence $(\beta_j)_{j\in\N_0}:= (\sigma_{k(j)})_{j\in\N_0}$ determined by (\ref{beta})  with    $B^{\sigma,N}_{p,q} = B^\beta_{p,q}$.
Furthermore by Theorem
\ref{functionseq} we can construct for each sequence
$(b_j)_{j \in \N} \subset \C$ with $|b_j| \leq c_r 2^{jr}$ (for some $r>0$) a distribution
$$
W(x_1,\ldots,x_n) := \sum_{j=1}^\infty b_j e^{i2^jx_1}
\vspace{-3mm}
$$
such that
$$\psi W \in B^\beta_{p,q} \; \Leftrightarrow \; (\beta_k b_k)_{k \in \N} \in \ell_q\, , \quad \mbox{for any given }\, \psi \in {\cal S} \setminus \{ 0 \} ~.$$
If we assume $\; B^{\sigma,N}_{p,q} = B^{\beta}_{p,q}\subset L_1^{\rm loc}$, then it follows $\psi W \in L_1^{\rm loc}(\R^n)$ whenever $(\beta_k b_k)_{k \in \N} \in \ell_q$. With a choice of $\psi$ different from 0 everywhere, then also $W \in L_1^{\rm loc}(\R^n)$ and, consequently, the one variable version $w$ (that is, $w(t) := \sum_{j=1}^\infty b_j e^{i2^jt}$) is locally integrable too. In particular, $\sum_{j=1}^\infty b_j e^{i2^jt}$ is the Fourier series of a function in $L_1([0,2\pi])$ and by Proposition \ref{Kah} it follows $ (b_j)_{j \in \N} \in \ell_2$.

Since the assumption $(\beta_k b_k)_{k \in \N} \in \ell_q$ implies that $|b_j| \leq c_r 2^{jr}$ for some $r>0$, then we have shown that $ (b_j)_{j \in \N} \in \ell_2$ for all sequences $(b_k)_{k\in \N} \subset \C$ such that $(\beta_k b_k)_{k \in \N} \in \ell_q$. Given any $ (\gamma_j)_{j \in \N} \in \ell_{\frac{q}{2}}$ and defining
$$ b_j := |\gamma_j|^{\frac{1}{2}} \beta_j^{-1}~~~,$$
the assumption $(\beta_k b_k)_{k \in \N} \in \ell_q$ is satisfied and therefore $(\beta_k^{-2} |\gamma_k|)_{k \in \N} \in \ell_1$. If $q>2$, then  again by Proposition \ref{Inverse Hoelder} we have
$(\beta_k^{-2} )_{k \in \N} \in \ell_{(\frac{q}{2})'}$, i.e., $(\beta_k^{-1} )_{k \in \N} \in \ell_{\frac{2q}{q-2}}$ (with the understanding that $\ell_{\frac{2\infty}{\infty-2}}$ should be read as $\ell_2$).
\\ Finally, by  Proposition \ref{sigma-sigmak(j)} we can transfer this to the original sequence $\sigma$ with arbitrary $\sigma_0>0$ and obtain $(\sigma_k^{-1} )_{k \in \N_0} \in \ell_{\frac{2q}{q-2}}$.

\end{proof}

\begin{theo}
\label{1_nec_suf_F}
Let $0<p< \infty$, $0< q \leq \infty$. Let $N$ and $\sigma$ be admissible sequences with $N$ satisfying also Assumption \ref{assump}. The following are necessary and sufficient conditions for $\; F^{\sigma,N}_{p,q} \subset L_1^{\rm loc}$, where $\ell_{\frac{2\infty}{\infty-2}}$ should be understood as $\ell_2$:
\\[-1mm]
\begin{enumerate}
    \item $(\sigma_j^{-1} N_j^{n(\frac{1}{p}-1)})_{j \in \No} \in \ell_{\infty}$, \quad in case $\; 0 < p < 1\;$ and $\; 0<q \leq \infty$;
    \item $(\sigma_j^{-1})_{j \in \No} \in \ell_\infty$, \quad in case $\; 1 \leq p < \infty\;$ and $\; 0 < q \leq 2 $;
    \item $(\sigma_j^{-1})_{j \in \No} \in \ell_{\frac{2q}{q-2}}$, \quad in case $\; 1 \leq p < \infty\;$ and $\; 2 < q \leq \infty$.
\end{enumerate}
\end{theo}

\begin{proof}

(i) First we prove the sufficiency of the given conditions in each case.

In case 1, it follows directly from Proposition \ref{sufcond1}.

For both remaining cases we use, in sequence, Proposition \ref{Fcounterpart}, Corollary \ref{level zero} and Theorem \ref{classical}. This explains why we can write, assuming the condition in each one of the cases, that in case 2

$$F^{\sigma,N}_{p,q} \hookrightarrow F^{(1),N}_{p,q} = F^0_{p,q} \subset L_1^{\rm loc}$$

\noindent and in case 3

$$F^{\sigma,N}_{p,q} \hookrightarrow F^{(1),N}_{p,2} = F^0_{p,2} \subset L_1^{\rm loc}.$$

\bigskip

(ii) Now we prove the necessity of the conditions sated in cases 1 and 2.

If we assume $\; F^{\sigma,N}_{p,q} \subset L_1^{\rm loc}$,
by Proposition \ref{BF}
it follows
$$B^{\sigma,N}_{p,\min\{ p,q \}} \subset L_1^{\rm loc}~~.$$
In case $0 < p \le 1$ and $ 0 < q \le \infty$ it holds $ 0 <  \min\{ p,q \} \leq 1$ and
 by Theorem \ref{1_nec_suf_B}, part 1, we have $(\sigma_j^{-1} N_j^{n(\frac{1}{p}-1)})_{j \in \No} \in \ell_{\infty}$.
\\In case  $ 1 < p < \infty $ and $ 0 < q \le 2$ it holds $ 0 <  \min\{ p,q \} \le \min\{ p,2\}$ and
 by Theorem \ref{1_nec_suf_B}, part 2, we have $(\sigma_j^{-1})_{j \in \No} \in \ell_\infty$.
 \bigskip

(iii) Finally, the proof of the necessity of the condition sated in case 3 is the same, {\em mutatis mutandis}, as in the last part in Theorem
 \ref{1_nec_suf_B} because, under the conditions of Theorem \ref{functionseq},
 $$\psi W \in B^\beta_{p,q} \; \Leftrightarrow \; (\beta_k b_k)_{k \in \N} \in \ell_q \; \; \Leftrightarrow \;\psi W \in F^\beta_{p,q} ~.$$
\end{proof}

\begin{exa} {\em Let $\sigma_j:= 2^{sj} (1+j)^b$, where $b>0$
and $N_j = 2^j$. Then
$$B^{s_1}_{p,q} \hookrightarrow
B^{\sigma}_{p,q}
  \hookrightarrow   B^{s}_{p,q}$$
  for any $s_1 > s$, that is, we have a scale in smoothness finer
  than in the classical case.

  Naturally we obtain in some cases also really finer results
  concerning
  the embedding of  $B^{\sigma}_{p,q}$ in $L_1^{\rm loc}$.
\\[1mm] (i) ~~Let $0<p\leq 1$ and $ 1 < q \le \infty$. Then the classical result
gives the embedding if and only if $s > n(\frac{1}{p} - 1)$, while in
our example the embedding is still true if $s = n(\frac{1}{p} -
1)$ and in addition $ b > \frac{q-1}{q}$.
\\[1mm] (ii)~~Let $1<p\le 2 $ and $ \min\{ p,2
\} < q \le \infty$. Then, in contrast to the classical case,
$s=0$ is possible if and only if  $b
> \frac{q-p}{pq}$ (meaning $b > \frac{1}{p}$ if $q=\infty$).
\\[1mm] (iii) Let $2<p\le \infty $ and $  \min\{ p,2
\} < q \le \infty$. Then again  $s=0$ is possible if and only if $b
> \frac{q-2}{2q}$ (meaning $b > \frac{1}{2}$ if $q=\infty$).
\\[1mm] (iv) The same is true for the  $F$-spaces in the
case $1\le p <
  \infty $ and $ 2 < q \le\infty$, where instead of $s> 0 $ now $s =
  0$ together with $b >\frac{q-2}{2q}$ (meaning $b > \frac{1}{2}$ if $q=\infty$) is permitted.}
\end{exa}

The following extends \cite[Cor. 3.3.1]{ST95} to our setting:

\begin{cor}
\label{corollary}
Let $N$ and $\sigma$ be admissible sequences with $N$ satisfying
also Assumption \ref{assump}.
\\[2mm](i)~Let $0<p < \infty \,,\, 0 < q \leq \infty$.
The following two assertions are equivalent:
\\[2mm] \hspace*{1cm} $~~~~~~\; B^{\sigma,N}_{p,q} \subset L_1^{\rm loc}$
\\ and
\\\hspace*{1cm} $\; ~~~~~~B^{\sigma,N}_{p,q} \hookrightarrow L_{\max\{1,p\}}$ 
.
\\[2mm](ii)~Let $\, 0 < q \leq \infty$.
The following two assertions are equivalent:
\\[2mm] \hspace*{1cm} $~~~~~~\; B^{\sigma,N}_{\infty,q} \subset L_1^{\rm loc}$
\\ and
\\\hspace*{1cm} $\; ~~~~~~B^{\sigma,N}_{\infty,q} \hookrightarrow bmo \;$.
\\[2mm] (iii)~Let $0<p < \infty \,,\, 0 < q \leq \infty$. The following two assertions are equivalent:
\\[2mm]\hspace*{1cm} $~~~~~~\; F^{\sigma,N}_{p,q} \subset L_1^{\rm loc}$
\\ and
\\ \hspace*{1cm}$~~~~~~\; F^{\sigma,N}_{p,q} \hookrightarrow L_{\max\{1,p\}}$ 
.

\end{cor}

\begin{proof}
Since the implication in which one concludes that $B^{\sigma,N}_{p,q}$ or $F^{\sigma,N}_{p,q}$ is in $L_1^{\rm loc}$ is obvious, we concentrate on the reverse one. So, let us assume that $\; B^{\sigma,N}_{p,q} \subset L_1^{\rm loc}$ when proving (i) and (ii) above and that $\; F^{\sigma,N}_{p,q} \subset L_1^{\rm loc}$ when proving (iii).

In what follows we shall use the following classical facts without further notice:
\medskip

\begin{tabular}{ll}
$F^{0}_{p,2}=L_p, \quad 1<p<\infty$ & (\cite[The. 2.5.6(i)]{Tri83}); \\[3mm]
$F^{0}_{1,2}= h_1$ & (\cite[The. 2.5.8/1]{Tri83}); \\[3mm]
$h_1\hookrightarrow L_1$ & (\cite[Rem. 2.5.8/4]{Tri83}); \\[3mm]
$F^{0}_{\infty,2}= bmo$ & (\cite[The. 2.5.8/2]{Tri83}); \\[3mm]
$B^0_{\infty,2} \hookrightarrow F^{0}_{\infty,2}$ & (cf. \cite[Prop. 2.3.2/2(iii), The. 2.11.2]{Tri83}).
\end{tabular}

\bigskip

(i) The $B$ case when $0<p<\infty$.

\medskip

First let $\; 0 < p \leq 1\;$ and $\; 0<q \leq \infty$.
\\We have, by Theorem \ref{1_nec_suf_B}, that $(\sigma_j^{-1}
N_j^{n(\frac{1}{p}-1)})_{j \in \No} \in \ell_{q'}$ and by
Proposition \ref{CF-t1} and Corollary \ref{level zero} it follows
$$B^{\sigma,N}_{p,q} \hookrightarrow B^{(1),N}_{1,1} = B^{0}_{1,1} = F^{0}_{1,1} \hookrightarrow  F^{0}_{1,2} = h_1 \hookrightarrow L_1 = L_{\max\{1,p\}}\;.$$

In case $\; 1 < p < \infty\;$ and $\; 0 < q \leq \min\{ p,2 \}$
Theorem \ref{1_nec_suf_B} implies $(\sigma_j^{-1})_{j \in \No} \in
\ell_\infty$ and by Proposition \ref{BF}, Proposition
\ref{Fcounterpart} and  Corollary \ref{level zero} we have
$$\; B^{\sigma,N}_{p,q} \hookrightarrow F^{\sigma,N}_{p,2} \hookrightarrow
F^{(1),N}_{p,2} = F^{0}_{p,2}= L_p = L_{\max\{1,p\}}\;.$$

If $\; 1< p \leq 2\;$ and $\; \min\{ p,2 \} < q \leq \infty$, then Theorem \ref{1_nec_suf_B} implies
$(\sigma_j^{-1})_{j \in \No} \in \ell_{\frac{pq}{q-p}}$ and
combining Proposition \ref{CF-t1}, Corollary \ref{level zero} and
Proposition \ref{Fcounterpart} we get
$$\; B^{\sigma,N}_{p,q} \hookrightarrow B^{(1),N}_{p,p} = B^{0}_{p,p} = F^{0}_{p,p} \hookrightarrow F^{0}_{p,2}= L_p = L_{\max\{1,p\}}\;.$$

Finally in case $\; 2 < p < \infty\;$ and $\; \min\{ p,2 \} < q
\leq \infty$
  Theorem \ref{1_nec_suf_B} gives $(\sigma_j^{-1})_{j \in \No} \in
 \ell_{\frac{2q}{q-2}}$ and again Proposition \ref{CF-t1}, Corollary \ref{level zero} and Proposition \ref{BF} lead to
$$\; B^{\sigma,N}_{p,q} \hookrightarrow B^{(1),N}_{p,2} = B^{0}_{p,2}  \hookrightarrow F^{0}_{p,2}= L_p = L_{\max\{1,p\}}\;.$$

(ii) The $B$ case when $p=\infty$.

\medskip

Then we have, similarly as above, that,
in case $\; 0 < q \leq \min\{ p,2 \}$, $(\sigma_j^{-1})_{j \in
\No} \in \ell_\infty$ and
$$\; B^{\sigma,N}_{\infty,q} \hookrightarrow
B^{(1),N}_{\infty,2} = B^{0}_{\infty,2}\hookrightarrow
F^{0}_{\infty,2}= bmo  \;;$$in case $\; \min\{ p,2
\} < q \leq \infty$, $(\sigma_j^{-1})_{j \in \No} \in
\ell_{\frac{2q}{q-2}}$ and
$$\; B^{\sigma,N}_{\infty,q} \hookrightarrow
B^{(1),N}_{\infty,2} = B^{0}_{\infty,2}\hookrightarrow
F^{0}_{\infty,2}= bmo  \;.$$

(iii) The $F$ case.

\medskip

Let first  $\; 0 < p < 1\;$ and $\; 0<q \leq \infty$. Then by
Theorem \ref{1_nec_suf_F} it holds $(\sigma_j^{-1}
N_j^{n(\frac{1}{p}-1)})_{j \in \No} \in \ell_{\infty}$. By Theorem
\ref{BFB} we obtain
$$F^{\sigma,N}_{p,q} \hookrightarrow B^{\sigma'',N}_{1,p} \quad \mbox{ with }  \; \sigma''_j = \sigma_j N_j^{n(1 - \frac{1}{p})}\;.$$
 Moreover
by Proposition \ref{CF-t1} and  Corollary \ref{level zero} we get
$$B^{\sigma'',N}_{1,p} \hookrightarrow B^{(1),N}_{1,1} = B^{0}_{1,1} = F^{0}_{1,1} \hookrightarrow  F^{0}_{1,2} = h_1 \hookrightarrow L_1 = L_{\max\{1,p\}}\;.$$

If $\; 1 \leq p < \infty\;$ and $\; 0 < q \leq 2 $
 we obtain $(\sigma_j^{-1})_{j \in \No} \in \ell_\infty$ and by
 Proposition \ref{Fcounterpart} and Corollary \ref{level zero} we have
$$ \; F^{\sigma,N}_{p,q}   \hookrightarrow F^{(1),N}_{p,2} = F^{0}_{p,2}= L_p = L_{\max\{1,p\}} ~~~~\mbox{in case }  1 < p$$
and
$$\; F^{\sigma,N}_{1,q}   \hookrightarrow F^{(1),N}_{1,2} = F^{0}_{1,2}= h_1 \hookrightarrow L_1 = L_{\max\{1,p\}} ~~~~\mbox{in case }  p = 1~.$$

At last, in case $\; 1 \leq p < \infty\;$ and $\; 2 < q \leq
\infty$ we get $(\sigma_j^{-1})_{j \in \No} \in
\ell_{\frac{2q}{q-2}}$ and in a similar way by Proposition
\ref{Fcounterpart} and Corollary \ref{level zero}
$$\; F^{\sigma,N}_{p,q}   \hookrightarrow F^{0}_{p,2} \hookrightarrow  L_p = L_{\max\{1,p\}} ~.$$
\end{proof}

\begin{rem}
{\em It follows from the preceding proof that, as in the classical case, we also have in case $0 < p \le 1$  that
\\[2mm] \hspace*{1cm} $\; B^{\sigma,N}_{p,q} \subset L_1^{\rm loc}~~$
if and only if   $\; ~~~B^{\sigma,N}_{p,q} \hookrightarrow h_1$
 \\and
\\[2mm]  \hspace*{1cm} $\; F^{\sigma,N}_{p,q} \subset L_1^{\rm loc}~~~~$
if and only if   $\; ~~~F^{\sigma,N}_{p,q} \hookrightarrow h_1\;.$}
\end{rem}


\end{document}